\documentclass[12pt]{article}

\usepackage{amssymb, amsmath, amsthm}
\usepackage[all]{xypic}

\usepackage{palatino}
\usepackage[mathcal]{euler}

\theoremstyle{plain}
    \newtheorem{theorem}{Theorem}[section]
    \newtheorem{lemma}[theorem]{Lemma}
    \newtheorem{corollary}[theorem]{Corollary}
    \newtheorem{proposition}[theorem]{Proposition}

 \theoremstyle{definition}
    \newtheorem{definition}[theorem]{Definition}
    
    \newtheorem{example}[theorem]{Example}

    \newtheorem{remark}[theorem]{Remark}

\DeclareMathOperator{\Hom}{Hom}
\DeclareMathOperator{\End}{End}
\DeclareMathOperator{\Ad}{Ad}
\DeclareMathOperator{\Res}{Res}
\DeclareMathOperator{\SO}{SO}
\DeclareMathOperator{\Spin}{Spin}
\DeclareMathOperator{\Crit}{Crit}

\DeclareMathOperator{\DInd}{D-Ind}
\DeclareMathOperator{\DRes}{D-Res}

\DeclareMathOperator{\se}{se}
    
\begin{document}

\title
 {Formal geometric quantisation for proper~actions} 

\author{Peter Hochs and Varghese Mathai}




\newcommand{\kg}{\mathfrak{g}}
\newcommand{\kk}{\mathfrak{k}}
\newcommand{\kp}{\mathfrak{p}}
\newcommand{\gse}{\mathfrak{g}^*_{\mathrm{se}}}
\newcommand{\kse}{\mathfrak{k}^*_{\mathrm{se}}}
\newcommand{\Ldom}{\Lambda_+}

\newcommand{\C}{\mathbb{C}}
\newcommand{\R}{\mathbb{R}}
\newcommand{\Z}{\mathbb{Z}}
\newcommand{\N}{\mathbb{N}}

\newcommand{\HH}{\mathcal{H}}
\newcommand{\cO}{\mathcal{O}}

\newcommand{\Qform}{Q^{-\infty}}
\newcommand{\Mom}{(M, \omega)}
\newcommand{\Nnu}{(N, \nu)}

\newcommand{\DIndstar}{\bigl(\DInd_K^G\bigr)^*}

\newcommand{\KK}{K \! K}

\maketitle

\begin{abstract}
We define formal geometric quantisation for proper Hamiltonian actions by possibly noncompact groups on possibly noncompact, prequantised symplectic manifolds, generalising work of Weitsman and Paradan. We study the functorial properties of this version of formal geometric quantisation, and relate it to a recent result by the authors via a version of the shifting trick. For (pre)symplectic manifolds of a certain form, quantisation commutes with reduction, in the sense that formal quantisation equals a more direct version of quantisation.
\end{abstract}

\tableofcontents


\section*{Introduction}

Consider  a Hamiltonian action by a compact Lie group $K$ on a possibly noncompact prequantised symplectic manifold $(N, \nu)$, with proper momentum map. Weitsman \cite{Weitsman} defined the \emph{formal geometric quantisation} of this action, which by definition commutes with reduction:
\[
\Qform_K(N, \nu) = \sum_{\lambda \in \Ldom^K} Q(N_{\lambda}) [\pi^K_{\lambda}],
\]
where $\Ldom^K$ is the set of dominant integral weights of $K$, with respect to a maximal torus and positive root system, and $\pi^K_{\lambda}$ is the irreducible representation of $K$ with highest weight $\lambda \in \Ldom^K$. 
For such $\lambda$, $N_{\lambda}$ is the symplectic reduction of the given action at $\lambda/i$ (which is a symplectic orbifold if $\lambda/i$ is a regular value of the momentum map).
Formal geometric quantisation takes values in the abelian group
\[
R^{-\infty}(K) = \Hom_{\Z}(R(K), \Z),
\]
where $R(K)$ is the representation ring of $K$.

Paradan \cite{Paradan1} proved that formal geometric quantisation is functorial with respect to Cartesian products and restriction to subgroups. These two properties imply that it is compatible with the $R(K)$-module structure on $R^{-\infty}(K)$. Ma and Zhang \cite{MaZ,MaZ2}, and also Paradan \cite{Paradan2} proved that quantisation commutes with reduction, in the sense that
\[
Q_K(N, \nu) = \Qform_K(N, \nu),
\]
for a certain definition of the quantisation $Q_K(N, \nu) \in R^{-\infty}(K)$.

On the other hand, Landsman \cite{Landsman} proposed a definition of geometric quantisation of a Hamiltonian action by a Lie group $G$ on a prequantised symplectic manifold $(M, \omega)$, if the orbit space $M/G$ is compact.
He used the \emph{analytic assembly map} from the Baum--Connes conjecture \cite{BCH}, which takes values in the \emph{$K$-theory} group $K_*(C^*_{(r)}G)$ of the (full or reduced) $C^*$-algebra of the group $G$. Applying this assembly map to a Dirac operator coupled to a prequantum line bundle yields Landsman's definition of
\begin{equation} \label{eq def quant Landsman}
Q_G(M, \omega) \in K_*(C^*_{(r)}G).
\end{equation}
Mathai and Zhang \cite{MZ} showed that Landsman's version of quantisation commutes with reduction at the trivial representation, at least if one multiplies the symplectic form $\omega$ by a large enough integer. For (possibly only presymplectic) manifolds of the form $M=G\times_K N$, with $N$ a compact prequantised Hamiltonian $K$-manifold, it was shown in \cite{HochsPresympl} that
\begin{equation} \label{eq [Q,R]=0 cocpt}
Q_G(M, \omega) = \sum_{\lambda \in \Ldom^K} Q(M_{\lambda + \rho_c}, \omega_{\lambda + \rho_c})[\lambda].
\end{equation}
Here $[\lambda]$ is a certain generator of $K_*(C^*_rG)$, $\rho_c$ is half the sum of the compact positive roots, and  $(M_{\lambda + \rho_c}, \omega_{\lambda + \rho_c})$ is the symplectic reduction of the action at $(\lambda + \rho_c)/i$. The shift over $\rho_c$ appears because $\Spin^c$-quantisation is used rather than Dolbeault-quantisation.

A common generalisation of $R^{-\infty}(K)$ and $K_*(C^*_rG)$ is the \emph{$K$-homology} group $K^*(C^*_rG)$ of $C^*_rG$. In view of the successes for quantisation with values in $R^{-\infty}(K)$ and $K_*(C^*_rG)$, it makes sense to find a definition of quantisation with values in $K^*(C^*_rG)$, without assuming the group or the orbit space to be compact. In this note, we generalise the formal quantisation studied by Weitsman and Paradan to noncompact groups. 
For manifolds of the form $M = G\times_K N$ as in \eqref{eq [Q,R]=0 cocpt}, but with $N$ possibly noncompact, 
we define geometric quantisation in $K^*(C^*_rG)$, and show that it equals formal quantisation. This means that this version of quantisation commutes with reduction (Theorem \ref{thm [Q,R]=0}). We study the functorial properties of formal quantisation, and give a relation with the main result in \cite{HM} via a version of the shifting trick.

\subsection*{Acknowledgements}

The authors would like to thank Paul--\'Emile Paradan for helpful comments.

\section{Compact groups}

Let $K$ be a compact, connected  Lie group, with Lie algebra $\kk$. Let $\Ldom^K$ be the set of dominant integral weights of $K$, with respect to a maximal torus and a choice of positive roots. For $\lambda \in \Ldom^K$, let $\pi^K_{\lambda}$ be the irreducible representation of $K$ with highest weight $\lambda$.
Consider the abelian group
\[
R^{-\infty}(K) := \Hom_{\Z}(R(K), \Z)
\]
Here $R(K)$ denotes the representation ring of $K$. Note that $R^{-\infty}(K)$ is generated by the elements $[\pi^K_{\lambda}]^*$, for $\lambda \in \Ldom^K$, where
\[
[\pi^K_{\lambda}]^*\bigl([\pi^K_{\lambda'}]\bigr) = \delta_{\lambda \lambda'} 
	:= \left\{ \begin{array}{ll} 	1 & \text{if $\lambda = \lambda'$;} \\
						0 & \text{if $\lambda \not= \lambda'$,}\end{array}\right.
\]
for all $\lambda' \in \Ldom^K$.

Let $(N, \nu)$ be a prequantised symplectic manifold, equipped with a Hamiltonian $K$-action. Suppose the momentum map $\Phi^K: N \to \kk^*$ is proper. Then for every $\lambda \in \Ldom^K$, the symplectic reduction \cite{MW}   $N_{\lambda}$ of the action at $\lambda/i$ is compact. Hence it has a quantisation $Q(N_{\lambda}) \in \Z$, where one can use Meinrenken and Sjamaar's approach \cite{MS} in the singular case.

Weitsman \cite{Weitsman} introduced the \emph{formal geometric quantisation} $Q^{-\infty}_K(N, \nu)$ of the action by $K$ on $(N, \nu)$, as
\[
Q^{-\infty}_K(N, \nu) = \sum_{\lambda \in \Ldom^K} Q(N_{\lambda}) [\pi^K_{\lambda}]^* \quad \in R^{-\infty}(K).
\]
Paradan \cite{Paradan1} proved that  formal quantisation is functorial with respect to restriction to subgroups, and also notes that it is functorial with respect to Cartesian products. 

To state Paradan's result on restriction to a subgroup $K'<K$, we consider the
abelian group
 $R^{-\infty}(K)_{K'}$ of formal differences of  equivalence classes of representations of $K$ whose restrictions to $K'$ decompose into irreducible representations of $K'$ with finite multiplicities. One has
\[
R(K) \subset R^{-\infty}(K)_{K'} \subset R^{-\infty}(K),
\]
and
\[
\begin{split}
R^{-\infty}(K)_{K}&= R^{-\infty}(K);\\
R^{-\infty}(K)_{\{e\}} &= R(K).
\end{split}
\]
Here we identify $R^{-\infty}(K) = \Hom_{\Z}(R(K), \Z)$ with the abelian group of formal differences of equivalence classes of representations of $K$ containing finitely many copies of all irreducibles, via the map $[\pi^K_{\lambda}]^* \mapsto [\pi^K_{\lambda}]$. By definition of $R^{-\infty}(K)_{K'}$, there is a well-defined restriction map
\[
\Res^K_{K'}: R^{-\infty}(K)_{K'} \to R^{-\infty}(K').
\]

Let  $\kk'$ be the Lie algebra of $K'$, and suppose that the momentum map $\Phi^{K'}: N \to (\kk')^*$ for the action by $K'$ on $N$ is still proper.  A criterion for this condition is given in Proposition 2.11 in \cite{ParadanHolDS}. Then Paradan showed that $Q^{-\infty}_K(N, \nu)  \in R^{-\infty}(K)_{K'}$, and 
\begin{equation} \label{eq restr cpt}
\Res^K_{K'} \bigl(Q^{-\infty}_K(N, \nu) \bigr) = Q^{-\infty}_{K'}(N, \nu).
\end{equation}

In addition, for $j = 1,2$, let $K_j$ be a compact, connected Lie group, and let 
 $(N_j, \nu_j)$ be a prequantised Hamiltonian $K_j$-manifold with proper momentum map. Then Paradan points out that
\begin{equation} \label{eq mult cpt}
Q^{-\infty}_{K_1}(N_1, \nu_1) \otimes Q^{-\infty}_{K_2}(N_1, \nu_1) = Q^{-\infty}_{K_1 \times K_2}(N_1 \times N_2, \nu_1 \times \nu_2).
\end{equation}
Note that $R^{-\infty}(K)$ is an $R(K)$-module, via the tensor product of representations. 
The properties \eqref{eq restr cpt} and \eqref{eq mult cpt} together imply that, if $N_1$ is compact,
\begin{equation} \label{eq quant ring str}
Q^{-\infty}_{K}(N_1, \nu_1) \cdot Q^{-\infty}_{K}(N_2, \nu_2) = Q^{-\infty}_{K}(N_1 \times N_2, \nu_1 \times \nu_2),  
\end{equation}
where the dot $\cdot$ denotes the $R(K)$-module structure of $R^{-\infty}(K)$. 

Our goal is to generalise the definition of formal geometric quantisation, and its functoriality properties with respect to restriction and Cartesian products, to noncompact groups. Then $R^{-\infty}(K)$ will be replaced by $K$-homology of group $C^*$-algebras.

\section{$K$-homology of group $C^*$-algebras and formal quantisation}

Let $G$ be a connected Lie group containing $K$ as a maximal compact subgroup. Let $\kg$ be the Lie algebra of $G$. Let $C^*_rG$ be the reduced $C^*$-algebra of $G$. We will write $d:= \dim(G/K)$. The \emph{Connes--Kasparov conjecture}, proved for almost connected groups by Chabert, Echterhoff and Nest \cite{CEN}, states that there is an isomorphism of Abelian groups
\[
\DInd_K^G: R(K) \xrightarrow{\cong} K_d(C^*_rG).
\]
This isomorphism is called \emph{Dirac induction}, and is given by
\begin{equation} \label{eq def DInd}
\DInd_K^G[\pi^K_{\lambda}] = \mu_{G/K}^G \bigl[D^{\lambda}_{G/K} \bigr],
\end{equation}
for $\lambda \in \Ldom^K$, where $\mu_{G/K}^G$ is the analytic assembly map \cite{BCH}, and $D^{\lambda}_{G/K}$ is a Dirac operator on $G/K$ coupled to the representation $\pi^K_{\lambda}$.

Let $K^d(C^*_rG)$ be the \emph{$K$-homology} group of $C^*_rG$ in degree $d$.
Since $K_d(C^*_rG) \cong R(K)$ is torsion-free, the universal coefficient theorem \cite{RS} implies that
\begin{equation} \label{eq UCT}
K^d(C^*_rG) \cong \Hom_{\Z}(K_d(C^*_rG), \Z).
\end{equation}
In particular, $R^{-\infty}(K) = K^0(C^*_rK)$.
The isomorphism \eqref{eq UCT} is given by the Kasparov product. Pulling back along the Dirac induction map defines an isomorphism of Abelian groups
\begin{equation} \label{eq pullback DInd}
\bigl(\DInd_K^G\bigr)^*: K^d(C^*_rG) \xrightarrow{\cong} R^{-\infty}(K).
\end{equation}

For $\lambda \in \Ldom^K$, we write $[\lambda]$ for the generator $\DInd_K^G[\pi^{K}_{\lambda}]$ of $K^d(C^*_rG)$. Let $[\lambda]^* \in K^d(C^*_rG)$ be the corresponding generator, defined by
\[
[\lambda]^*\bigl([\lambda']\bigr) = \delta_{\lambda \lambda'},
\]
for $\lambda' \in \Ldom^K$. Then
\begin{equation} \label{eq Dindstar lambda}
\bigl(\DInd_K^G\bigr)^* [\lambda]^* = [\pi^K_{\lambda}]^*.
\end{equation}

We consider $K_{d}(C^*G)$ as a subgroup of $K^d(C^*_rG)$ via the map $[\lambda] \mapsto [\lambda]^*$.


Using the generators $[\lambda]^*$ of $K^d(C^*_rG)$, one can generalise formal geometric quantisation to actions by noncompact groups as follows. Let $(M, \omega)$ be a prequantised symplectic manifold, equipped with a proper Hamiltonian $G$-action. Suppose the momentum map $\Phi^G: M \to \kg^*$ is \emph{$G$-proper}, in the sense that the inverse image of every cocompact set is cocompact. (By a cocompact set we mean  a set with compact quotient by the group action.) Then all symplectic reductions of the action are compact. Assume that all symplectic reductions at elements of $\Lambda^K_+$ have well-defined quantisations; see below for a discussion of this assumption.
\begin{definition}
The \emph{formal geometric quantisation} of the action by $G$ on $(M, \omega)$ is
\[
Q^{-\infty}_G(M, \omega) = \sum_{\lambda \in \Lambda^K_+} Q(M_{\lambda}) [\lambda]^* \quad \in K^d(C^*_rG).
\]
\end{definition}


Let $\xi \in \kg^*$ be a value of $\Phi^G$. Any of the following conditions is sufficient for the symplectic reduction $M_{\xi}$ to have a well-defined quantisation $Q(M_{\xi}) \in \Z$.
\begin{enumerate}
\item If $G$ is \emph{compact}, one can define $Q(M_{\xi})$ using Meinrenken and Sjamaar's methods \cite{MS}.
\item Suppose that
\begin{enumerate}
\item $\xi$ is a \emph{quasi-regular value} of $\Phi^G$, in the sense that all $G$ orbits in $(\Phi^G)^{-1}(G\cdot \xi)$ have the same dimension; and
\item the prequantum line bundle $L\to M$ is \emph{almost equivariantly locally trivial at level $\xi$}, in the sense that for all $m \in (\Phi^G)^{-1}(G\cdot \xi)$, the identity component of the stabiliser $G_m$ acts trivially on the fibre $L_m$.
\end{enumerate}
Then $M_{\xi}$ is a compact symplectic orbifold, with a prequantum line bundle induced by $L$, and can be quantised. See Section 2.2 in \cite{MS}.
\item As a special case of the second point, suppose that all points  in $(\Phi^G)^{-1}(G\cdot \xi)$ have conjugate stabilisers, and that $G\cdot \xi$ is locally closed. Then by Theorem 16 in \cite{BL}, the symplectic reduction $M_{\xi}$ is smooth. 
\item Suppose $G$ is semisimple with discrete series, and let  $K<G$ be a maximal compact subgroup.
If the stabiliser $G_{\xi}$ is \emph{compact}, i.e.\ $\xi$ is \emph{strongly elliptic}, set
\[
\widetilde{M} := (\Phi^{G})^{-1}(\kg^*_{\se}),
\]
where $\kg^*_{\se}$ is the set of strongly elliptic elements. By Corollary 2.4 and Proposition 2.6 in \cite{Weinstein}, this is a nonempty open subset of $\kg^*$. Hence $\widetilde{M}$ is a $G$-invariant open neighbourhood of $(\Phi^G)^{-1}(G\cdot \xi)$ in $M$. By Propositions 2.8 and 2.14 in \cite{HochsDS},  $\widetilde{M}$ has the form
\[
 \widetilde{M} \cong G\times_K N,
\]
where $N:= (\Phi^G)^{-1}(\kk^*) \cap \widetilde{M}$ is a Hamiltonian $K$-manifold, with momentum map defined by the restriction of $\Phi^G$. Hence there is a homeomorphism
\[
M_{\xi} =  \widetilde{M}_{\xi} = N_{\xi}, 
\]
which is a symplectomorphism of $\xi$ is a regular value of $\Phi^K$. Hence $M_{\xi} \cong N_{\xi}$ has a well-defined quantisation by the first point. 

The generators $[\lambda]$, for $\lambda \in \Lambda^K_+ \cap i\kg^*_{\se}$ correspond to discrete series representations, see Remark 2.5 in \cite{HochsDS}, and also Example \ref{eq ds DRes}.

\item As a special case of the fourth point, suppose that $\Phi^G$ is \emph{proper}, rather than just $G$-proper. Then it was pointed out in \cite{ParadanHolDS} that $\Phi^G(M) \subset \kg^*_{\se}$, so the fourth point applies to any value $\xi$ of $\Phi^G$.
\end{enumerate}

\section{Quantisation commutes with reduction}

\emph{Quantisation commutes with reduction} is the statement that
\begin{equation} \label{eq [Q,R]=0 G}
Q_G(M, \omega) = \Qform_G(M, \omega),
\end{equation}
for some definition of $Q_G(M, \omega) \in K^d(C^*_rG)$. For a prequantised Hamiltonian $K$-manifold  $(N, \nu)$ with proper momentum map $\Phi^K$, as considered earlier, such a definition was given by Ma and Zhang \cite{MaZ,MaZ2} and Paradan \cite{Paradan2}. They defined $Q_K(N, \nu)$ via expanding (relatively) compact subsets of $N$. Braverman's  index theory for generalised Dirac operators on possibly noncompact manifolds \cite{Braverman} can be applied to give a direct analytic definition of quantisation, provided the critical point set of the norm-squared function of $\Phi^K$ is compact (which it is if $N$ is real-algebraic and $\Phi^K$ is algebraic, as noted in Lemma 3.24 of \cite{ParadanHolDS}. 
Ma and Zhang, and also Paradan, proved that 
\begin{equation} \label{eq [Q,R]=0 K}
Q_K(N, \nu) = \Qform_K(N, \nu).
\end{equation}

A definition of $Q_G(M, \omega) \in K^d(C^*_rG)$ satisfying \eqref{eq [Q,R]=0 G} can be given if $M$ is of a particular form. Suppose that $G$ is semisimple, and let $\kg = \kk \oplus \kp$ be a Cartan decomposition. Suppose $M$ is of the form $M = G\times_K N$ considered in \cite{HochsDS,HochsPresympl}, the quotient of $G \times N$ by the $K$-action
\[
k\cdot (g, n) = (gk^{-1}, kn),
\]
for $k \in K$, $g \in G$ and $n \in N$.
As in \cite{HochsPresympl}, consider the $G$-invariant presymplectic form (i.e.\ closed two-form) $\omega$ on $M$ given by
\[
\omega_{[e, n]}\bigl(Tq(X+v), Tq(Y+w)\bigr) := \nu_n(v, w) - \langle\Phi^K(n), [X,Y]\rangle,
\]
where $n \in N$, $X, Y \in \kp$, $v, w \in T_nN$, and $q: G\times N \to M$ is the quotient map.  
 In general, $\omega$ may be degenerate, but all constructions relevant to quantisation and reduction still apply. 
The momentum map $\Phi^G: M \to \kg^*$, given by
\begin{equation} \label{eq def Phi G}
\Phi^G[g, n] = \Ad^*(g) \Phi^K(n),
\end{equation}
for $g \in G$ and $n \in N$, is $G$-proper if $\Phi^K$ is proper.

If $G$ has discrete series representations and $\Phi^K(N) \subset \kk^* \hookrightarrow \kg^*$ lies inside the set $\kg^*_{\se}$ of strongly elliptic elements, then $\omega$ is an actual symplectic form (see \cite{HochsDS}, Proposition 2.4, with more details given in Proposition 12.4 in \cite{HochsThesis}). Conversely, any Hamiltonian $G$-manifold $(M, \omega)$ whose momentum map takes values in the strongly elliptic set is of the form $M = G\times_K N$ as above (\cite{HochsDS}, Proposition 2.14).

In \cite{HochsDS,HochsPresympl}, it was shown that for compact $N$, \emph{quantisation commutes with induction}, in the sense that
\[
Q_G(M, \omega) = \DInd_K^G\bigl(Q_K(N, \nu)\bigr).
\]
Here $Q_G(M, \omega) \in K^*_d(C^*_rG)$ is Landsman's version \eqref{eq def quant Landsman} of geometric quantisation.
This shows that the following definition reduces to Landsman's definition if $N$ is compact.
\begin{definition} \label{def quant ind}
For $M = G\times_K N$ as above, the \emph{geometric quantisation} of the action by $G$ on $M$ is
\[
Q_G(M, \omega)  = \left(\DIndstar\right)^{-1} \bigl(Q_K(N, \nu) \bigr) \quad \in K^d(C^*_rG).
\]
\end{definition}
The formal quantisation of the action by $G$ on $(M, \omega)$ depends on the precise procedure used to quantise the reduced spaces $M_{\xi}$. Let $Q(M_{\xi})$ be defined as the usual index of a Dirac operator if $\xi$ is a quasi-regular value of $\Phi^G$, and as $Q(N_{\xi})$ if $\xi$ is a singular value (since $M_{\xi} \cong N_{\xi}$).
\begin{theorem}[Quantisation commutes with reduction] \label{thm [Q,R]=0}
If $M$ is of the form $M = G \times_K N$, then \eqref{eq [Q,R]=0 G} holds.
\end{theorem}
\begin{proof}
The equality \eqref{eq [Q,R]=0 K} implies that
\[
\begin{split}
Q_G(M, \omega) &=  \left(\DIndstar\right)^{-1} \bigl( Q_K(N, \nu) \bigr) \\
	&= \left(\DIndstar\right)^{-1}  \bigl( \Qform_K(N, \nu) \bigr).
\end{split}
\]
Because of Lemma \ref{lem quant ind} below, the latter expression equals $\Qform_G(M, \omega)$.
\end{proof}

\begin{lemma} \label{lem quant ind}
In this setting, one has
\[
\DIndstar \bigl(\Qform_G(M, \omega) \bigr) = \Qform_K(N, \nu).
\]
\end{lemma}
\begin{proof}
Let $\lambda \in \Ldom^K$. If $\lambda/i$ is a singular value of $\Phi^G$, one by definition has 
\begin{equation}  \label{eq reds equal}
Q(M_{\lambda}) = Q(N_{\lambda}).
\end{equation}
For quasi-regular values, one has an isomorphism of (pre)symplectic orbifolds
\[
(M_{\lambda}, \omega_{\lambda}) \cong (N_{\lambda}, \nu_{\lambda})
\]
(see \cite{HochsPresympl}, Lemma 5.1), so that, in particular, $(M_{\lambda}, \omega_{\lambda})$ is actually symplectic, rather than just presymplectic. Hence \eqref{eq reds equal} also holds in that case.
The desired equality therefore follows from \eqref{eq Dindstar lambda}.
\end{proof}

Theorem \ref{thm [Q,R]=0} motivates the search for a definition of $Q_G(M, \omega) \in K^d(C^*_rG)$ for arbitrary $(M, \omega)$, generalising Definition \ref{def quant ind}. 

\section{A restriction map}

We return to the case where $G$ is any connected Lie group. Let $G' < G$ be a closed, connected subgroup that has a maximal compact subgroup $K'$ contained in $K$. We write $d' := \dim(G'/K') $. Set
\[
K^*(C^*_rG)_{G'} := \Bigl( \bigl(\DInd^G_K \bigr)^*\Bigr)^{-1}\bigl(R^{-\infty}(K)_{K'} \bigr).
\]
(Note that since all maximal compact subgroups are conjugate, the ring $R^{-\infty}(K)_{K'}$ is independent of the choice of maximal compact subgroup $K'<G'$.)
\begin{definition}
The \emph{Dirac restriction} map $\DRes^G_{G'}$ is defined by commutativity of the following diagram:
\[
\xymatrix{
K^d(C^*_rG)_{G'} \ar[r]^-{\DRes^G_{G'}} \ar[d]_{(\DInd^{G}_{K})^*}^{\cong} & K^{d'}(C^*_rG') \ar[d]^{(\DInd^{G'}_{K'})^*}_{\cong} \\
R^{-\infty}(K)_{K'} \ar[r]^-{\Res^K_{K'}} & R^{-\infty}(K').
}
\]
\end{definition}
Because the Dirac restriction map is modelled on the restriction map from $K$ to $K'$, it may not contain all representation theoretic information concerning restriction of representations from $G$ to $G'$. It does have natural functoriality properties with respect to formal quantisation, as we will see.

\begin{example} \label{eq ds DRes}
Suppose $G$ is semisimple with discrete series. Then $d$ is even. Let $\rho_c$ be half the sum of the positive roots of $K$. 
Let $\lambda \in \Ldom^K$, and suppose $\lambda$ is strongly elliptic. Let $\pi^G_{\lambda}$ be the irreducible discrete series representation of $G$ with Harish--Chandra parameter $\lambda + \rho_c$.  Then $\pi^G_{\lambda}$ defines a $K$-theory class
\[
[\pi^G_{\lambda}] \in K_0(C^*_rG)
\]
(see \cite{Lafforgue}, Section 2.2). In (5.3) in \cite{HochsPresympl}, it is noted that
\[
[\pi^G_{\lambda}] = (-1)^{d/2}[\lambda] = (-1)^{d/2}\DInd_K^G[\pi^K_{\lambda}].
\]
Hence the image of $[\pi^G_{\lambda}]$ in $K^d(C^*_rG)$ is
 $[\pi^G_{\lambda}]^* := (-1)^{d/2}[\lambda]^* \in K^0(C^*_rG)$, and
\[
\DIndstar \bigl([\pi^G_{\lambda}]^* \bigr)= (-1)^{d/2}[\pi^K_{\lambda}]^*.
\]

Let $\Ldom^{K'}$ be the set of dominant integral weights of $K'$ with respect to a maximal torus and positive roots, compatible with the choices made for $K$. Write
\[
\Res^K_{K'} (\pi^K_{\lambda}) = \sum_{\lambda' \in \Ldom^{K'}} m_{\lambda'} \pi^{K'}_{\lambda'},
\]
for certain integer coefficients $m_{\lambda'}$. Then
\[
\begin{split}
\bigl(\DInd^{G'}_{K'}\bigr)^* \circ \DRes^{G}_{G'}[\pi^G_{\lambda}]^* &= \Res^K_{K'} \circ \bigl(\DInd^{G}_{K}\bigr)^* [\pi^G_{\lambda}]^* \\
	&= (-1)^{d/2} \sum_{\lambda' \in \Ldom^{K'}} m_{\lambda'} [\pi^{K'}_{\lambda'}]^*.
\end{split}
\]
Hence
\[
\DRes^{G}_{G'}[\pi^G_{\lambda}]^* = (-1)^{d/2} \sum_{\lambda' \in \Ldom^{K'}} m_{\lambda'} [\lambda']^*,
\]
by \eqref{eq Dindstar lambda}.
\end{example}


Suppose that $G$ and $G'$ are semisimple, and that $M$ is of the form $M = G \times_K N$ as above. Then formal geometric quantisation has the following functoriality property with respect to Dirac restriction.
\begin{proposition} \label{prop restr}
Suppose that the momentum map for the action by $G'$ on $M$ is still proper. Then (omitting the various symplectic forms from the notation)
we have
\[
\Qform_{G}(G\times_K N) \in K^d(C^*_rG)_{G'},
\]
 and 
\[
\DRes^G_{G'} \bigl( \Qform_G (G \times_K N) \bigr)= \Qform_{G'}(G' \times_{K'} N).
\]
\end{proposition}
\begin{proof}
Because of the form \eqref{eq def Phi G} of the momentum map for the action by $G'$ on $M$, this map is $G'$-proper if and only if the momentum map for the action by $K'$ on $N$ is proper. Hence, by Paradan's result \eqref{eq restr cpt}, one has $\Qform_K(N) \in R^{-\infty}(K)_{K'}$, and
\[
\Res^K_{K'} \bigl(\Qform_K(N) \bigr) = \Qform_{K'}(N). 
\]
Lemma \ref{lem  quant ind} states that
\[
\Qform_K(N)  = \bigl(\DInd_{K}^{G} \bigr)^*  \bigl( \Qform_{G}(G\times_K N) \bigr). 
\]
Hence $\Qform_{G}(G\times_K N) \in K^d(C^*_rG)_{G'}$, and 
\[
\begin{split}
\bigl(\DInd_{K'}^{G'} \bigr)^* \circ \DRes^G_{G'} \bigl( \Qform_{G}(G\times_K N) \bigr)&= \Res^K_{K'} \circ \bigl(\DInd_{K}^{G} \bigr)^*  \bigl( \Qform_{G}(G\times_K N) \bigr)\\
	&=  \Res^K_{K'} \bigl(\Qform_K(N) \bigr)\\
	&= \Qform_{K'}(N) \\
	&= \bigl(\DInd_{K'}^{G'} \bigr)^* \bigl( \Qform_{G'}(G' \times_{K'}N) \bigr).
\end{split}
\]
\end{proof}

\begin{remark}
One would expect a restriction map $\Res^G_{G'}$ to satisfy
\[
\Res^G_{G'} \bigl( \Qform_G(M)\bigr) = \Qform_{G'}(M),
\]
compare with Theorem D in \cite{ParadanHolDS}. Proposition \ref{prop restr} reflects a different nature of the Dirac restriction map.
\end{remark}

\section{Products of generators}

In Section 5.3 of \cite{HochsDS}, a multiplicativity property of the analytic assembly map is discussed. This will allow us to generalise the multiplicative property \eqref{eq mult cpt} of formal geometric quantisation to noncompact groups. That in turn leads to a generalisation of property \eqref{eq quant ring str} of formal quantisation with respect to the $R(K)$-module structure of $R^{-\infty}(K)$.

Let $G_1$ and $G_2$ be locally compact groups, acting properly and cocompactly on locally compact Hausdorff spaces $X_1$ and $X_2$, respectively. There are Kasparov product maps on equivariant $K$-homology and on $K$-theory,
\[
\begin{split}
K^{G_1}_*(X_1) \times K^{G_2}_*(X_2) &\xrightarrow{\times} K^{G_1 \times G_2}_*(X_1 \times X_2); \\
K_*(C^*_rG_1) \times K_*(C^*_rG_2) &\xrightarrow{\times} K_*(C^*_r(G_1 \times G_2)).
\end{split}
\]
By Theorem 5.2 in \cite{HochsDS}, the assembly maps $\mu_{X_j}^{G_j}$ and $\mu_{X_1 \times X_2}^{G_1 \times G_2}$ satisfy
\begin{equation} \label{eq mult ass}
\mu_{X_1}^{G_1}(a_1) \times \mu_{X_2}^{G_2}(a_2) = \mu_{X_1 \times X_2}^{G_1 \times G_2}(a_1 \times a_2),
\end{equation}
for all $a_j \in K^{G_j}_*(X_j)$, at least if $X_1$ and $X_2$ are metrisable.

Now suppose $G_1$ and $G_2$ are connected, semisimple Lie groups. Let $K_j < G_j$ be maximal compact subgroups, 
 and suppose that the adjoint representations $\Ad: K_j \to \SO(\kp_j)$ lift to $\Spin(\kp_j)$, for Cartan decompositions $\kg_j = \kk_j \oplus \kp_j$.
  (This is always true for certain covers of the groups $G_j$.) Write $d_j := \dim(G_j/K_j)$.
\begin{lemma} \label{lem mult Kthry}
Let $\lambda_j \in \Ldom^{K_j}$. Then one has
\[
[\lambda_1] \times [\lambda_2] = [(\lambda_1, \lambda_2)] \quad \in K_{d_1 + d_2}(C^*_r(G_1 \times G_2)).
\]
(Note that $\Ldom^{K_1\times K_2} = \Ldom^{K_1} \times \Ldom^{K_2}$.)
\end{lemma}
\begin{proof}
Let  $K_j < G_j$ be as above. In this setting, for $G = G_j$ and $K=K_j$, the Dirac operator $D^{\lambda}_{G/K}$ used in the definition \eqref{eq def DInd} of Dirac induction is defined explicitly as follows.
Let $\{X_1, \ldots, X_n\}$ be a basis of $\kp$, orthonormal with respect to the Killing form. Let $\Delta_{\kp}$ be the standard representation of $\Spin(\kp)$, and let $c: \kp \to \End(\Delta_{\kp})$ be the Clifford action. Let $\lambda \in \Ldom^K$, and let $V_{\lambda}$ be the representation space of $\pi^K_{\lambda}$. 
Then
\begin{equation} \label{eq def D lambda}
D^{\lambda}_{G/K} = \sum_{j=1}^n X_j \otimes c(X_j) \otimes 1_{V_{\lambda}}
\end{equation}
on
\[
\bigl(C^{\infty}(G) \otimes \Delta_{\kp} \otimes V_{\lambda}\bigr)^K.
\]

 In $K^{G_1 \times G_1}_*\bigl((G_1 \times G_2)/(K_1 \times K_2)\bigr)$, it follows from \eqref{eq def D lambda}  that
 for all $\lambda_j \in \Ldom^{K_j}$,
\[
\begin{split}
\bigl[ D^{(\lambda_1, \lambda_2)}_{(G_1 \times G_2)/(K_1 \times K_2)}  \bigr] 
	&= \bigl[ D^{\lambda_1}_{G_1/K_1} \otimes 1 + 1 \otimes D^{\lambda_2}_{G_2/K_2}   \bigr] \\
	&=    \bigl[ D^{\lambda_1}_{G_1/K_1} \bigr] \times \bigl[ D^{\lambda_2}_{G_2/K_2} \bigr]
\end{split}
\]
Here we have used the fact that $\pi^{K_1 \times K_2}_{(\lambda_1, \lambda_2)} = \pi^{K_1}_{\lambda_1} \otimes \pi^{K_2}_{\lambda_2}$. We conclude that, because of \eqref{eq mult ass},
\[
\begin{split}
 [(\lambda_1, \lambda_2)] &= \DInd^{G_1 \times G_2}_{K_1 \times K_2}[\pi^{K_1 \times K_2}_{(\lambda_1, \lambda_2)}] \\
 	&= \mu_{(G_1 \times G_2)/(K_1 \times K_2)}^{G_1 \times G_2} \bigl[ D^{(\lambda_1, \lambda_2)}_{(G_1 \times G_2)/(K_1 \times K_2)}  \bigr] \\
	&= \mu_{(G_1 \times G_2)/(K_1 \times K_2)}^{G_1 \times G_2} \left(  \bigl[ D^{\lambda_1}_{G_1/K_1} \bigr] \times \bigl[ D^{\lambda_2}_{G_2/K_2} \bigr]  \right) \\
	&= \mu_{G_1/K_1}^{G_1} \bigl[ D^{\lambda_1}_{G_1/K_1} \bigr] \times  \mu_{G_2/K_2}^{G_2} \bigl[ D^{\lambda_2}_{G_2/K_2} \bigr] \\
	&= [\lambda_1]  \times [\lambda_2].
\end{split}
\]
\end{proof}

We will use an extension of Lemma \ref{lem mult Kthry} to an equality involving the Kasparov product map on $K$-homology
\begin{equation} \label{eq Kasp Khom}
K^*(C^*_rG_1) \times K^*(C^*_rG_2) \xrightarrow{\times} K^*(C^*_r(G_1 \times G_2)).
\end{equation}
\begin{corollary} \label{cor mult Khom} 
For all $\lambda_j \in \Ldom^{K_j}$, one has
\[
[\lambda_1]^* \times [\lambda_2]^* = [(\lambda_1, \lambda_2)]^* \quad \in K^{d_1 + d_2}(C^*_r(G_1 \times G_2)).
\]
\end{corollary}
\begin{proof}
Let $\lambda_j, \mu_j \in \Ldom^{K_j}$. Then
\begin{equation} \label{eq lambda mu}
 [(\lambda_1, \lambda_2)]^* \bigl( [(\mu_1, \mu_2)] \bigr) = \delta_{\lambda_1 \mu_1} \delta_{\lambda_2 \mu_2} = 
 	[\lambda_1]^*\bigl( [\mu_1] \bigr) \cdot [\lambda_2]^*\bigl( [\mu_2] \bigr).
\end{equation}
The isomorphism \eqref{eq UCT} is induced by the Kasparov product, i.e.\  for $\lambda \in \Ldom^K$, the homomorphism
\[
[\lambda]^*: \KK(\C, C^*_rG) \to \Z
\]
is given by taking the Kasparov product with $[\lambda]^* \in \KK(C^*_rG, \C)$. Hence the right hand side of \eqref{eq lambda mu} equals
\[
\bigl( [\mu_1]\times_{C^*_rG_1} [\lambda_1]^* \bigr) \cdot \bigl( [\mu_2]\times_{C^*_rG_2} [\lambda_2]^* \bigr) = 
	\bigl( [\mu_1] \times [\mu_2] \bigr) \times_{C^*_rG_1 \otimes C^*_rG_2} \bigl( [\lambda_1]^* \times [\lambda_2]^* \bigr),
\]
where we have used the associativity properties of the Kasparov product.
By Lemma \ref{lem mult Kthry}, the latter expression equals
\[
 [(\mu_1, \mu_2)] \times_{C^*_rG_1 \otimes C^*_rG_2} \bigl( [\lambda_1]^* \times [\lambda_2]^* \bigr) =
	 \bigl( [\lambda_1]^* \times [\lambda_2]^* \bigr) \bigl( [(\mu_1, \mu_2)] \bigr).
\]
\end{proof}

\section{Module structures}

Corollary \ref{cor mult Khom} implies that formal quantisation is multiplicative. For $j = 1,2$, let $(M_j, \omega_j)$ be equivariantly prequantised proper Hamiltonian $G_j$-manifolds, with $G_j$-proper momentum maps. Suppose the groups $G_j$ are connected and semisimple.
\begin{corollary} \label{cor mult}
One has
\[
\Qform_{G_1 \times G_2}(M_1 \times M_2, \omega_1 \times \omega_2) = \Qform_{G_1}(M_1, \omega_1) \times \Qform_{G_2}(M_2, \omega_2) \quad \in K^{d_1+d_2}(C^*_r(G_1 \times G_2)). 
\]
\end{corollary}
\begin{proof}
Let $\lambda_j \in \Ldom^{K_j}$. As noted by Paradan \cite{Paradan1}, one has an equality of symplectic reductions
\[
\bigl((M_1 \times M_2)_{(\lambda_1, \lambda_2)}, (\omega_1 \times \omega_2)_{(\lambda_1, \lambda_2)} \bigr) \cong \bigl( (M_1)_{\lambda_1} \times (M_2)_{\lambda_2}, (\omega_1)_{\lambda_1} \times (\omega_2)_{\lambda_2} \bigr),
\]
if $\lambda_1/i$ and $\lambda_2/i$ are regular values of the respective momentum maps.

Since the manifolds $(M_j)_{\lambda_j}$ are compact, one has
\[
Q\bigl((M_1)_{\lambda_1}\bigr) Q\bigl((M_2)_{\lambda_2} \bigr) = 
Q\bigl((M_1)_{\lambda_1} \times (M_2)_{\lambda_2} \bigr) \quad \in \Z.
\]
Hence, because $\Ldom^{K_1 \times K_2} = \Ldom^{K_1} \times \Ldom^{K_2} =: \Ldom$, Corollary \ref{cor mult Khom} implies that
\[
\begin{split}
\Qform_{G_1 \times G_2}(M_1 \times M_2, \omega_1 \times \omega_2) 
	&= \sum_{\scriptscriptstyle(\lambda_1, \lambda_2) \in   \Ldom} Q\bigl( (M_1 \times M_2)_{(\lambda_1, \lambda_2)}  \bigr)  [(\lambda_1, \lambda_2)]^* \\
& = \sum_{\scriptscriptstyle(\lambda_1, \lambda_2) \in  \Ldom}  Q\bigl((M_1)_{\lambda_1} \bigr) Q\bigl((M_2)_{\lambda_2} \bigr) [\lambda_1]^* \times [\lambda_2]^* \\
&= \Qform_{G_1}(M_1, \omega_1) \times \Qform_{G_2}(M_2, \omega_2). 
\end{split}
\]
\end{proof}

The compatibility property \eqref{eq quant ring str} of formal quantisation with the $R(K)$-module structure on $R^{-\infty}(K)$ can be generalised to noncompact groups. It is possible to equip $K^d(C^*_rG)$ with a $K_d(C^*_rG)$-module structure in the following way.  For $a \in K_d(C^*_rG)$ we have $a^* \in K^d(C^*_rG)$, via the inclusion map defined by $[\lambda] \mapsto [\lambda]^*$ on generators. If $b \in K^d(C^*_rG)$, then we have
\[
a^* \times b \in K^0(C^*_r(G\times G)),
\]
where $\times$ denotes the Kasparov product \eqref{eq Kasp Khom}. Corollary \ref{cor mult Khom} implies that
\begin{multline*}
\bigl(\DInd_{K \times K}^{G \times G} \bigr)^* (a^* \times b ) = \bigl(\DInd_K^G \bigr)^* (a^*) \otimes \bigl(\DInd_K^G \bigr)^* (b)  \\
\in R(K) \otimes R^{-\infty}(K) \subset R(K\times K)_{\Delta(K)}.
\end{multline*}
where $\Delta(K) < K \times K$ is the diagonal subgroup.  So
\[
a^* \times b \in K^0(C^*_r (G \times G) )_{\Delta(G)},
\]
which is the domain of the Dirac restriction map
\[
\DRes^{G \times G}_{\Delta(G)} :K^0(C^*_r (G \times G) )_{\Delta(G)} \to K^d(C^*_rG).
\]
\begin{definition} \label{def mod str}
The $K_d(C^*_rG)$-module structure of $K^d(C^*_rG)$ is defined by
\begin{equation} \label{eq mod str}
a\cdot b := \DRes^{G \times G}_{\Delta(G)}(a^* \times b),
\end{equation}
for $a$ and $b$ as above.
\end{definition}

\begin{lemma} \label{lem iso mod}
Dirac induction is compatible with the $R(K)$-module structure of $R^{-\infty}(K)$ and the  $K_d(C^*_rG)$-module structure of $K^d(C^*_rG)$, in the sense that
\[
\bigl(\DInd_K^G \bigr)^* \bigl(  \DInd_K^G[\pi] \cdot b \bigr) = [\pi] \cdot \bigl(\DInd_K^G \bigr)^*(b),
\]
for all finite-dimensional representations $\pi$ of $K$, and all $b \in K^d(C^*_rG)$.
\end{lemma}
\begin{remark}
This lemma in particular implies that \eqref{eq mod str} indeed defines a module structure.
\end{remark}
\noindent
\emph{Proof of Lemma \ref{lem iso mod}.}
It is enough to check the equality for irreducible $\pi = \pi^K_{\lambda_1}$ and generators $b = [\lambda_2]^*$ of $K^d(C^*_rG)$, for $\lambda_j \in \Ldom^K$. Then, using Corollary \ref{cor mult Khom} and  \eqref{eq Dindstar lambda}, one finds that
\[
\begin{split}
\bigl(\DInd_K^G \bigr)^* \bigl(  \DInd_K^G[\pi] \cdot b \bigr) 
&= \bigl(\DInd_K^G \bigr)^* \circ \DRes^{G \times G}_G \bigl( [\lambda_1]^* \times [\lambda_2]^* \bigr) \\
&= \bigl(\DInd_K^G \bigr)^* \circ \DRes^{G \times G}_G [(\lambda_1, \lambda_2)]^* \\
&= \Res^{K \times K}_K \circ \bigl(\DInd_{K\times K}^{G\times G} \bigr)^* [(\lambda_1, \lambda_2)]^* \\
&= \Res^{K \times K}_K \bigl[ \pi^{K\times K}_{(\lambda_1, \lambda_2)} \bigr]^* \\
&= \Res^{K \times K}_K \bigl[ \pi^K_{\lambda_1} \otimes \pi^K_{\lambda_2}\bigr]^* \\
&= [\pi^K_{\lambda_1}] \cdot [\pi^K_{\lambda_2}]^* \\
&= [\pi] \cdot \bigl(\DInd_K^G \bigr)^*(b).
\end{split}
\]
\hfill $\square$
\medskip

Proposition \ref{prop restr},  Corollary \ref{cor mult} and Lemma \ref{lem iso mod} have the following consequence. Let  $(N_j, \nu_j)$ be
prequantised Hamiltonian $K$-manifolds. Suppose $N_1$ is compact, and the momentum map for the action by $K$ on $N_2$ is proper. Then, with respect to the module structure of Definition \ref{def mod str},
\[
\Qform_G(G\times_K N_1) \cdot \Qform_G(G\times_K N_2) = \Qform_G\bigl(G\times_K(N_1 \times N_2)\bigr).
\]

\section{The shifting trick}

As in \cite{HM}, consider a $G$-invariant metric on the trivial bundle $M \times \kg^* \to M$, equipped with the $G$-action
\[
g\cdot (m, \xi) = (g\cdot m, \Ad^*(g)\xi),
\]
for $g \in G$, $m \in M$ and $\xi \in \kg^*$. Denote the induced norm on the fibre at $m$ by $\|\cdot\|_m$. Let $\HH$ be the associated norm-squared function of the momentum map $\Phi^G$:
\[
\HH(m) = \|\Phi^G(m)\|_m^2.
\]
Consider the one-form $d_1\HH \in \Omega^1(M)$ defined by
\[
(d_1\HH)_m = d_m\bigl(m' \mapsto   \|\Phi^G(m')\|_m^2\bigr).
\]
Let $\Crit_1(\HH)$ be the set of zeroes of $d_1\HH$. Under the assumptions that $\Crit_1(\HH)/G$ is compact and $G$ is unimodular, the invariant quantisation
\[
Q(M, \omega)^G \quad \in \Z
\]
was defined in \cite{HM}. It was proved that for $p \in \N$ large enough,
\begin{equation} \label{eq [Q,R]=0 zero}
Q(M, p\omega)^G = Q(M_0, p\omega_0),
\end{equation}
and conjectured that this equality holds for $p = 1$.

Let $\lambda \in \Ldom^K$. Consider the reduction map $R^G_{\lambda}: K^d(C^*_rG) \to \Z$ given by taking the multiplicity of $[\lambda]^*$ (i.e.\ by applying elements to $[\lambda]$). Let $\cO_{\lambda}^-:= \Ad^*(G)\lambda/i$ be the coadjoint orbit through $\lambda/i$, equipped with minus the standard Kirillov--Kostant--Souriau symplectic form. Let $\omega^{\lambda} \in \Omega^2(M \times \cO_{\lambda}^-)$ be the induced product symplectic form.
Let $\HH_{\lambda}$ be the function $\HH$ defined above, for the diagonal action by $G$ on $M \times \cO_{\lambda}^-$. Suppose the conjecture that \eqref{eq [Q,R]=0 zero} holds for $p=1$ is true, which is the case for example if the action is free. 
Then one has the following version of the shifting trick.
\begin{proposition} \label{prop shift}
If $\Crit_1(\HH_{\lambda})/G$ is compact, then, if $\lambda/i$ is a regular value of $\Phi^G$, 
\[
R^G_{\lambda} \bigl( \Qform_G(M, \omega) \bigr)= Q\bigl(M\times \cO_{\lambda}^-,  \omega^{\lambda}\bigr)^G.
\] 
\end{proposition}
\begin{proof}
See Corollary 5.12 in \cite{HM}.
\end{proof}

Under the stronger assumption (which may be restrictive) that $\Crit_1(\HH_{\lambda})/G$ is compact  for \emph{all} $\lambda$,
one can define a \emph{semi-formal} version of quantisation as
\[
Q^{\text{semi}}_G(M, \omega) =  \sum_{\lambda \in \Ldom^K} Q\bigl(M\times \cO_{\lambda}^-,  \omega^{\lambda}\bigr)^G [\lambda]^*.
\]
This version of quantisation has the advantage that it is well-defined regardless of how singular the reduced spaces $M_{\lambda}$ may be. In the special case where $\lambda/i$ is a regular value of $\Phi^G$ for all $\lambda$, then Proposition \ref{prop shift} implies that semi-formal quantisation commutes with reduction, in the sense that $Q^{\text{semi}}_G(M, \omega) = \Qform_G(M, \omega)$. (But note that this is only true if \eqref{eq [Q,R]=0 zero} holds for $p=1$.)


\begin{thebibliography}{99}
%
%
\bibitem{BL} L.\ Bates and E.\ Lerman, `Proper group actions and symplectic stratified spaces', {\em Pacific J.\ Math.} 181 (1977), no.\ 2,  201--229.
\bibitem{BCH} {  P.\ Baum, A.\ Connes and N.\ Higson}, `Classifying space for proper actions and $K$-theory of group $C^*$-algebras', {\em Contemp. Math.} 167 (1994) 241--291.
\bibitem{Braverman} {  M.\ Braverman}, `Index theorem for equivariant Dirac operators on noncompact manifolds', {\em $K$-Theory} 27 (2002), no.\ 1, 61--101. 
\bibitem{CEN} {  J.\ Chabert, S.\ Echterhoff and R.\ Nest}, `The Connes--Kasparov conjecture for almost connected groups and for linear $p$-adic groups',  {\em Publ.\ Math.\ Inst.\ Hautes \'Etudes Sci.}  97 (2003) 239--278.
\bibitem{HR} {  N.\ Higson and J.\ Roe}, {\em Analytic $K$-homology}, Oxford mathematical monographs (Oxford university press, New York, 2000).
\bibitem{HochsThesis} {  P.\ Hochs}, {\em Quantisation commutes with reduction for cocompact Hamiltonian group actions}, Ph.D.\ thesis, Radboud University (2008) ISBN 978-90-9022607.
\bibitem{HochsDS} {  P.\ Hochs}, `Quantisation commutes with reduction at discrete series representations of semisimple Lie groups', {\em Adv. Math.} 222 (2009) 862--919.
\bibitem{HochsPresympl}  {  P.\ Hochs}, `Quantisation of presymplectic manifolds, $K$-theory and group representations', {\em Proc.\ Amer.\ Math.\ Soc.}, to appear, ArXiv:12110107.
\bibitem{HL} {  P.\ Hochs and N.P.\ Landsman}, `The Guillemin--Sternberg conjecture for noncompact groups and spaces', {\em J. $K$-theory} 1 (2008) 473--533.
\bibitem{HM} {  P.\ Hochs and V.\ Mathai}, `Geometric quantization and families of inner products', ArXiv:1309.6760.
\bibitem{Lafforgue} { V. Lafforgue}, `Banach $\KK$-theory and the Baum--Connes conjecture', {\em Proc. ICM Beijing} vol.\ 2 (2002) 795--812.
\bibitem{Landsman} {  N.P.\ Landsman}, `Functorial quantization and the Guillemin-Sternberg conjecture',
{\em Twenty years of Bialowieza: a mathematical anthology} (eds.\
S.\ Ali, G.\  Emch, A.\ Odzijewicz, M.\ Schlichenmaier and S.\ Woronowicz, World scientific, Singapore, 2005) 23--45.
\bibitem{MW} {  J.E.\ Marsden and A.\ Weinstein}, `Reduction of symplectic manifolds with symmetry',
{\em Rep. Math. Phys. }5 (1974) 121--130.
\bibitem{MaZ} {  X.\ Ma and W.\ Zhang}, `Geometric quantization for proper moment maps', {\em C. R. Acad. Sci. Paris}, Ser. I 347 (2009) 389--394.
\bibitem{MaZ2} 
{  X.\ Ma and W.\ Zhang}, `Geometric quantization for proper moment maps: the Vergne conjecture', {\em Acta Math.}, to appear,      ArXiv:0812.3989.
\bibitem{MZ}
{  V.\ Mathai and W.\ Zhang},  `Geometric quantization for proper actions' (with an appendix by U.\  Bunke),  {\em Adv.\ Math.}  225 (2010) 1224--1247 (ArXiv:0806.3138).
\bibitem{MS} {  E.\ Meinrenken and R.\ Sjamaar}, `Singular reduction and quantization', {\em Topology} 38 (1999) 699--762,
\bibitem{Paradan1} {  P.-E.\ Paradan}, `Formal geometric quantization', {\em Ann. Institut Fourier} 59 (2009) 199--238.
\bibitem{Paradan2} {  P.-E.\ Paradan}, `Formal geometric quantization II', {\em Pacific J. of Math.} 253 (2011)  169--212.
\bibitem{ParadanHolDS} P.-E.\ Paradan, `Quantization commutes with reduction in the non-compact setting: the case of the holomorphic discrete series', ArXiv:1201:5451.
\bibitem{RS} {  J.\ Rosenberg and C.\  Schochet} `The K\"unneth theorem and the universal coefficient theorem for Kasparov's generalized $K$-functor', {\em Duke Math.\ J.} 55 (1987) no.\ 2, 431--474.
\bibitem{Weinstein}  A.\ Weinstein, `Poisson geometry of discrete series orbits, and momentum convexity for noncompact group actions', {\em Lett.\ Math.\ Phys.} 56 (2001), no.\ 1,  17--30.
\bibitem{Weitsman} {  J.\ Weitsman}, `Non-Abelian symplectic cuts and the geometric quantization of noncompact
manifolds', EuroConf\'erence Mosh\'e Flato 2000, Part I (Dijon), {\em  Lett.\ Math.\ Phys.} 56 (2001) no. 1, 31--40.
\end{thebibliography}
\end{document}